\documentclass[12pt]{amsart}
\usepackage[top=1in, bottom =1in, left=1in, right = 1in]{geometry}

\usepackage{amsmath,amsthm,amssymb,amstext,amsxtra}
\usepackage[all]{xy}
\usepackage{mathrsfs}  
\usepackage{mathtools} 
\usepackage{colonequals}
\usepackage{enumerate}

\usepackage[backref,colorlinks=true, linkcolor=blue, citecolor=blue,urlcolor=blue]{hyperref}
\usepackage[nameinlink]{cleveref}

\newcommand{\linedef}[1]{\textsf{#1}}
\newcommand{\defi}[1]{\linedef{#1}}

\newcommand \A {\mathbb A}
\newcommand \C {\mathbb C}

\newcommand \F {\mathbb F}

\renewcommand \P {\mathbb P}
\newcommand \PP {\mathbb P}
\newcommand \Q {\mathbb Q}

\newcommand \Z {\mathbb Z}

\newcommand \Kbar {K^{\textup{al}}}
\newcommand \Qbar {\Q^{\textup{al}}}

\newcommand \eps {\varepsilon}

\newcommand{\sfE}{\mathsf{E}}

\newcommand \rhobar {\overline{\rho}}

\newcommand{\frakp}{\mathfrak{p}}

\DeclareMathOperator \End {End}
\DeclareMathOperator \Frob {Frob}
\DeclareMathOperator \Gal {Gal}
\DeclareMathOperator \GL {GL}
\DeclareMathOperator \img {img}
\DeclareMathOperator{\hht}{ht}

\DeclareMathOperator \Spec {Spec}
\DeclareMathOperator \Sym {Sym}

\DeclareMathOperator \Aut {Aut}

\DeclareMathOperator \opchar {char}

\DeclareMathOperator \Frac {{Frac}}

\DeclareMathOperator \bPic {\mathbf{Pic}}

\DeclareMathOperator{\ord}{ord}

\DeclareMathOperator{\SL}{SL}

\newcommand{\psmod}[1]{~(\textup{\text{mod}}~{#1})}

\numberwithin{equation}{subsection}

\theoremstyle{plain}
\newtheorem{thm}[equation]{Theorem}
\newtheorem{prop}[equation]{Proposition}

\newtheorem{lem}[equation]{Lemma}

\newtheorem{cor}[equation]{Corollary}

\theoremstyle{definition}

\newtheorem{cons}[equation]{Construction}

\theoremstyle{remark}
\newtheorem{remark}[equation]{Remark}

\newenvironment{enumalph}
{\begin{enumerate}}
{\end{enumerate}}

\newenvironment{enumroman}
{\begin{enumerate}}
{\end{enumerate}}

\title{On abelian varieties whose torsion is not self-dual}

\author{Sarah Frei}
\address{Department of Mathematics MS 136, Rice University, 6100 S.~Main St, Houston, TX 77005, USA}
\email{sarah.frei@rice.edu}

\author{Katrina Honigs}
\address{Department of Mathematics, Simon Fraser University, 8888 University Drive, Burnaby, British Columbia V5A 1S6, Canada}
\email{khonigs@sfu.ca}

\author{John Voight}
\address{Department of Mathematics, Dartmouth College, Kemeny Hall, Hanover, NH 03755, USA; Carslaw Building (F07), Department of Mathematics and Statistics, University of Sydney, NSW 2006, Australia}
\email{jvoight@gmail.com}

\date{\today}



\begin{document}

\begin{abstract}
We construct infinitely many abelian surfaces \(A\) defined over the rational numbers such that, for a prime $\ell \leqslant 7$, the $\ell$-torsion subgroup of \(A\) is not isomorphic as a Galois module to the $\ell$-torsion subgroup of its dual \(A\spcheck\). We do this by explicitly analyzing the action of the Galois group on the \(\ell\)-adic Tate module and its reduction modulo \(\ell\). 
\end{abstract}

\maketitle

\section{Introduction}

\subsection{Setup}

Let $K$ be a number field with algebraic closure $\Kbar$.  Let $A$ be an abelian variety over $K$ of dimension $g \colonequals \dim A \geq 1$. 
For $n \geq 1$, we obtain a representation of $\Gal_K \colonequals \Gal(\Kbar\,|\,K)$ 
\begin{equation} \label{eqn:rhoAbarn}
\rhobar_{A,n} \colon \Gal_K \to \Aut(A[n](\Kbar)) \simeq \GL_{2g}(\Z/n\Z).
\end{equation}
Here we compare the representation $\rhobar_{A,n}$ to the representation $\rhobar_{A\spcheck,n}$ associated with the dual abelian variety $A\spcheck \colonequals \bPic^0_A$.
The Weil pairing yields a canonical isomorphism
\begin{equation} \label{eqn:selfdual}
\rhobar_{A\spcheck,n} \cong \rhobar^{\ast}_{A,n} \otimes \varepsilon_n
\end{equation}
of Galois representations, where ${\phantom{\,}}^{\ast}$ denotes the contragredient representation and $\varepsilon_n$ is the cyclotomic character.

In general, these two linear representations are quite challenging to distinguish.  For most abelian varieties one encounters, 
there is an isomorphism $\rhobar_{A,n} \simeq \rhobar_{A\spcheck,n}$.  Indeed, if $A$ has a polarization $\lambda \colon A \to A\spcheck$ over $K$ whose degree is coprime to $n$---such as if $A$ is principally polarized over $K$
---then the polarization induces such an isomorphism.  In general, the number fields $K(A[n])$ and $K(A\spcheck[n])$ are always \emph{equal}, taken inside $\Kbar$ (\Cref{lem:samenumberfield}); so $\rho_{A,n}$ and $\rho_{A\spcheck,n}$ arise from representations of the same finite Galois group.  Of course, since $A$ and $A\spcheck$ are isogenous over $K$, they have isomorphic $\ell$-adic representations for all primes $\ell$ and hence the characteristic polynomials of $\rhobar_{A,n}(\sigma)$ and $\rhobar_{A\spcheck,n}(\sigma)$ agree for all $\sigma \in \Gal_K$.  In particular, for $n=\ell$ prime, the semi-simplifications of $\rhobar_{A,\ell}$ and $\rhobar_{A\spcheck,\ell}$ are also isomorphic (\Cref{lem:semisimp}). 

\subsection{Results}

Our main result shows that these representations need not be isomorphic in general.

\begin{thm}\label{thm:main}
Let $n \in \Z_{>0}$ be divisible by a prime $\ell \leqslant 7$.  Then there exist infinitely many pairwise geometrically non-isogenous abelian surfaces \(A\) over $\Q$ such that 
$\rhobar_{A,n} \not\simeq \rhobar_{A\spcheck,n}$. 
\end{thm}

Equivalently by \eqref{eqn:selfdual}, for a surface $A$ in \Cref{thm:main}, the representation $\rhobar_{A,n}$ is not self-dual up to twist by its similitude character, the cyclotomic character.  

It is enough to prove the theorem for $n=\ell \leqslant 7$ prime.  
We construct the abelian surfaces in \Cref{thm:main} by choosing elliptic curves $E_1$, $E_2$ with nontrivial $P\in E_1[\ell](\Q)$, $Q\in E_2[\ell](\Q)$ and gluing $E_1,E_2$ along the diagonal subgroup $\langle(P,Q)\rangle$.
The resulting abelian surfaces have a $(1,\ell)$-polarization but not a principal polarization over $\Q$. 
In fact, infinitely many of these surfaces do not have a principal polarization over $\Qbar$. We are able to prove the above theorem for odd values of $\ell$ by observing that although these abelian surfaces have a $\Q$-torsion point, their duals do not. In the $\ell=2$ case, the dual abelian surface will have a $\Q$-torsion point, but the Galois actions are, nevertheless, not isomorphic.

The underlying parameter space for our construction is essentially the product $Y_1(\ell) \times Y_1(\ell)$ of modular curves (with a slight reinterpretation required for $\ell=2$); for $\ell\leqslant 7$, this space is birational to $\A^2$.  We may therefore adjust the setup or ask for additional properties to be satisfied in \Cref{thm:main}.  For example, our results can be extended over any number field $K$ with $K \cap \Q(\zeta_\ell) = \Q$, see \cref{sec:mainproof}. On the other hand, for $\ell >7$, $Y_1(\ell)$ has genus greater than zero and fails to have infinitely many points over $\Q$; thus our construction cannot guarantee a distinction in $A[\ell]$ and $A\spcheck[\ell]$ over $\Q$.

Finally, we also go a bit further: forgetting the group structure, the linear representation $\rhobar_{A,n}$ yields a permutation representation $\pi_{A,n} \colon \Gal_K \to \Sym(A[n]) \simeq S_{n^{2g}}$.  If $\rhobar_{A,n} \simeq \rhobar_{A\spcheck,n}$ then of course $\pi_{A,n} \simeq \pi_{A\spcheck,n}$, but not necessarily conversely.  In fact, the abelian surfaces we exhibit to prove \Cref{thm:main} satisfy the stronger property 
that $\pi_{A,n} \not\simeq \pi_{A\spcheck,n}$
for $\ell\in \{3,5,7\}$.

\begin{cor}\label{cor:main}
Let \(\ell \in \{3,5,7\}\). Then there exist infinitely many geometrically nonisogenous abelian surfaces $A$ over $\Q$ such that $\pi_{A,\ell} \not\simeq \pi_{A\spcheck,\ell}$. Moreover, the linear representations $\Gal_K \to \GL_{\ell^{2g}}(k)$ induced by the permutation representations $\pi_{A,\ell}$ and $\pi_{A\spcheck,\ell}$ over any field \(k \) with $\opchar k = 0$ are not isomorphic.
\end{cor}

One could further consider subgroups $G \leqslant \GL_{2g}(\Z/n\Z)$ preserving a degenerate (but nonzero) alternating pairing up to scaling with the property that $G$ is not isomorphic to its contragredient twisted by the similitude character.  We classify these groups in the case $g=n=2$ in \Cref{prop:listofsubgroups}.  Attached to each $G$ would be an associated moduli space of polarized abelian varieties of dimension $g$, and the rational points of this moduli space which do not lift to the moduli space attached to any proper subgroup $G' < G$ would similarly give candidate examples.  \Cref{thm:main} can then be understood as exhibiting an explicit two-dimensional rational subspace for several such groups $G$.  

\subsection{Application}

The linear representation induced by the permutation representation associated to the $3$-torsion of an abelian surface $A$ over $K$ is contained in the $\ell$-adic \'etale cohomology of the generalized Kummer fourfold $K_2(A)$ \cite[Theorem~1.1]{Frei-Honigs} (see also Hassett--Tschinkel \cite[Proposition~4.1]{Hassett-Tschinkel}).  As a result \cite[Corollary~1.2]{Frei-Honigs}, the fourfolds $K_2(A)$ and $K_2(A\spcheck)$ are not derived equivalent \emph{over $K$} if the induced linear representations associated to $A[3]$ and $A\spcheck[3]$ are not isomorphic. 
Using the ideas of Huybrechts \cite[\S2.1]{Huybrechts} on twisted derived equivalence and cohomology, this result extends immediately to prove that under this condition, $K_2(A)$ and $K_2(A\spcheck)$ cannot be twisted derived equivalent, either.
In particular, \Cref{cor:main} (\Cref{prop:maincor}) implies that there are infinitely many abelian surfaces $A$ defined over $\Q$ where $K_2(A)$ and $K_2(A\spcheck)$ are not (twisted) derived equivalent over $\Q$; it would be interesting to determine if they 
have such a relationship over $K(A\spcheck[3])=K(A[3])$.

Also in the direction of derived equivalence, recall that, as seen in the proof of~\Cref{thm:main}, the abelian surfaces are such that $A[3](\Q)$ is nontrivial but $A\spcheck[3](\Q)$ is trivial. Since $A$ and $A\spcheck$ are derived equivalent \cite{Mukai}, this shows that the Mordell--Weil group is not a derived invariant. Note that the first dimension in which this could happen is for surfaces, since derived equivalent elliptic curves are isomorphic \cite[Theorem~1.1]{AKW}. 




\subsection{Contents}

In \cref{example.section} we exhibit our family of abelian surfaces, describe its basic properties, and complete the proof of \Cref{thm:main}. 
In \cref{sec:theproofs}, we give some further analysis, including a proof of 
\Cref{cor:main}, and conclude with some final remarks about related questions and future work.

\subsection*{Acknowledgements}

The authors would like to thank Eran Assaf, Asher Auel, Nils Bruin, Johan de Jong, Aaron Landesman, Pablo Magni, Bjorn Poonen, Ari Shnidman, Alexei Skorobogatov, David Webb, and Ariel Weiss for helpful comments.  Many thanks also to the anonymous referees for their feedback and corrections.  Frei was supported by an AMS-Simons Travel grant. Honigs was supported by an NSERC Discovery grant.
Voight was supported by grants from the Simons Foundation (550029 and SFI-MPS-Infrastructure-00008650).  

\section{Constructions and computations}\label{example.section}

We begin with the construction of the abelian surfaces $A$ used in \Cref{thm:main}.
We then compute the Galois action on \(A[\ell]\) and on \(A\spcheck [\ell]\) by comparing \(T_\ell A\) and \(T_\ell A\spcheck\) inside \(V_\ell A_0\), where \(A_0\) is a product of two elliptic curves and $A$ is isogenous to $A_0$. Finally, we give the proof of our main theorem. 

\subsection{Construction of the abelian surfaces}
\label{sec:construction}

Let $k$ be a field with absolute Galois group $\Gal_k \colonequals \Gal(k^{\textup{sep}}\,|\,k)$ 
and let $\ell\neq \opchar{k}$ be prime.  Recalling the introduction, a necessary but not sufficient condition for $A[\ell] \not\simeq A\spcheck[\ell]$ is that every polarization on $A$ has degree divisible by $\ell$.  We 
produce abelian surfaces satisfying this condition
by gluing together two (non-isogenous) elliptic curves along a subgroup of order $\ell$.  There are many references for this construction. For example it is described on MathOverflow \cite{MO}, implicitly suggested as an exercise \cite[Exercise 6.35]{Goren}, and recently exhibited \cite[Theorem 2.5]{BS}.  We present our own brief account, for completeness. We do not give the most general construction but address in \cref{sec:generalconst} how it can be generalized.
 

\begin{cons}\label{cons}
Let \(E_1\) and \(E_2\) be elliptic curves over \(k\) 
and let $P \in E_1[\ell](k)$ and $Q \in E_2[\ell](k)$ be $k$-rational points of order $\ell$.  Let
\[ G \colonequals \langle (P,Q) \rangle \leqslant E_1 \times E_2
\quad\text{and}\quad
A\colonequals (E_1\times E_2)/G, \]
and let $q \colon E_1 \times E_2 \to A$ be the quotient map.
\end{cons}

In \cref{sec:mainproof}, we will use \Cref{cons} in the proof of \Cref{thm:main}. 

\begin{lem}\label{lem:ppolarized}
With setup as in \Cref{cons}, the following statements hold.
\begin{enumalph}
\item $A$ is an abelian surface over $k$ with a $(1,\ell)$-polarization over $k$.
\item For a field extension $k' \supseteq k$, if there is no isogeny $E_1 \to E_2$ over $k'$, then any polarization on $A$ over $k'$ has degree divisible by $\ell$.  
\end{enumalph}
\end{lem}

\begin{proof}
First part (a).  Since $G$ is defined over $k$, so too is $A$ defined over $k$.  We have the principal product polarization $\lambda_0$ on $E_1 \times E_2$; it is a symmetric isogeny, so $\lambda_0\spcheck=\lambda_0$ under the canonical isomorphism of an abelian variety with its double dual.  Since $G$ is isotropic (the pairing is alternating and the group is cyclic) and $\lambda_0$ is a principal polarization, the pushforward under $q$ of $\lambda_0$ provides a $(1,\ell)$-polarization $\lambda$ on $A$. 
 
Next, part (b).  Without loss of generality, we may replace $k$ by $k'$.  Let $\lambda \colon A \to A\spcheck$ be a polarization (over $k$) of degree $d^2$.  Consider the pullback $q^*\lambda \colonequals q\spcheck \circ \lambda \circ q$, a polarization on $E_1 \times E_2$, so again $(q^*\lambda)\spcheck=q^* \lambda$.  The composition $\phi \colonequals \lambda_0^{-1} \circ q^*\lambda \in \End(E_1 \times E_2)$ is an endomorphism of degree $(\ell d)^2$.  Further, $\phi$ is fixed under the Rosati involution ${}^\dagger$ (associated to $\lambda_0$):
\begin{equation}
\phi^\dagger \colonequals \lambda_0^{-1} \phi\spcheck \lambda_0 = 
\lambda_0^{-1} (\lambda_0^{-1} \circ q^* \lambda)\spcheck \lambda_0 =
\lambda_0^{-1} (q^*\lambda) (\lambda_0^{-1}) \lambda_0 = \lambda_0^{-1} (q^*\lambda) = \phi.
\end{equation}
Since by hypothesis $E_1$ and $E_2$ are not isogenous over $k$, we have
\[ \End(E_1 \times E_2) \simeq \End(E_1) \times \End(E_2) \]
(endomorphisms over $k$).  The ring of Rosati-fixed endomorphisms of an elliptic curve is $\Z$---if the elliptic curve has complex multiplication, the Rosati involution acts by complex conjugation---so $\phi=(d_1,d_2)$ with $d_1,d_2 \in \Z_{>0}$ satisfying $d_1 d_2 = \ell d$. 
Since $\phi$ factors through $q$, we have the containment $\ker q \subseteq \ker \phi=E_1[d_1]\times E_2[d_2]$. 

Now assume for purposes of contradiction that $\ell \nmid d$. Without loss of generality, $\ell \mid d_1$ and $\ell \nmid d_2$, which implies that, under projection onto the $E_2$ factor, $\ker q$ projects to the trivial subgroup in $E_2[d_2]$. But this is a contradiction, since by construction $\ker q$ projects to a nontrivial subgroup under projection to both $E_1$ and $E_2$.
\end{proof}

\subsection{Background on Hilbert irreducibility}

In this section, we quickly adapt the statement of the Hilbert irreducibility theorem for our purposes.  For a reference, see Serre \cite[sections 9.2, 9.6]{Serre-lect} or \cite[Chapter 3]{serre2}, or Lang \cite[Chapter 9]{Lang}.




Let $K$ be a number field and let 
\begin{equation} 
f_t(x) = f(t_1,...,t_n;x) \in K(t_1,...,t_n)[x] 
\end{equation}
be an irreducible polynomial of degree $d$ over $K(t_1,\dots,t_n)$, the function field of $\A_K^n$ in the variables $t_1,\dots,t_n$.  The coefficients of $f_t(x)$ are simultaneously defined on a nonempty open subset $U \subseteq \A_K^n$ (avoiding denominators).  Suppose that $f_t(x)$ has generic Galois group $G \leq S_d$ over the field $K(t_1,\dots,t_n)$, obtained by the permutation action on the roots of $f_t(x)$ (in an algebraic closure).

\begin{thm}[Hilbert irreducibility theorem] \label{thm:HIT}
For all $a \in U(K)$ outside of a thin set, the specialization $f_a(x) \in K[x]$ has Galois group $G \leq S_d$ over $K$.
\end{thm}

\begin{proof}
The set of points where the Galois group is smaller is defined by nontrivial polynomial conditions, and so lies in a thin set: see e.g.\ Serre \cite[Proposition 3.3.5]{serre2}.  For further treatment, see also Serre \cite[Chapter 10]{Serre-lect}, Saltman \cite{Saltman}, or more recently Wittenberg \cite[section 1]{Wittenberg}.
\end{proof}


Since $K$ is Hilbertian, the set of points of $U(K)$ outside a thin set $Z \subseteq U$ has density $1$ in the sense that 
\begin{equation}
\lim_{B \to \infty} \frac{\#\{P \in (U \smallsetminus Z)(K) : \hht(P) \leq B\}}{\#\{P \in U(K) : \hht(P) \leq B\}} = 1
\end{equation}
where $\hht \colon U(K) \to \Q_{>0}$ is the usual height \cite[\S 9.7]{Serre-lect}.  

\Cref{thm:HIT} can be understood geometrically, as follows.  
%

\begin{cor} \label{cor:HITcor}
Let $X\to U$ be a finite morphism over $K$ with $X$ irreducible.  Then for points $t \in U(K)$ outside of a thin set, the fiber $X_t$ is irreducible over $K$.  If further $X \to U$ is generically Galois with generic Galois group $G \colonequals \Gal(K(X)\,|\,K(U))$, then for $a \in U(K)$ outside of a thin set the fiber $X_a \to \Spec K$ is also a $G$-Galois cover.  
\end{cor}

\begin{proof}
This is merely a restatement of \Cref{thm:HIT}.  Explicitly, if $U=\Spec A \subseteq \A_K^n$ is affine with $A \subseteq K(t_1,\dots,t_n)$ (a finitely generated $K$-subalgebra with fraction field equal to $K(t_1,\dots,t_n)$), we could take $X=\Spec A[x]/(f_t(x)) \to U$.  
\end{proof}



We will need to apply \Cref{thm:HIT} under multiple specializations: we will want not just that there are infinitely many $G$-extensions, but for these to be as disjoint as possible.  See also recent work of Zywina \cite{Zywina}, who studies families of abelian varieties with large Galois image in an effective manner.  



\begin{prop} \label{prop:G0G}
Let $X$ be irreducible and $X \to U$ a finite morphism that is generically Galois with group $G \colonequals \Gal(K(X)\,|\,K(U))$.  Let $L \supseteq K$ be the algebraic closure of $K$ in $K(X) \supseteq K(U)$.  Then the following statements hold.
\begin{enumalph}
\item $X$ has a canonical structure as a variety over $L$, corresponding to a factorization $X \to U_L \to U$ where $U_L \colonequals U \times_K L$.  
\item $X \times_L X$ is irreducible, and when considered as a $K$-scheme $(X \times_L X)_K$ it is irreducible.  
\item Let $G_0 \colonequals \Gal(L\,|\,K) \cong \Gal(L(U)\,|\,K(U))$ and let $\pi \colon G \to G_0$ be the restriction map.  Then the natural morphism $(X \times_L X)_K \to U \times_K U$ is a finite generically Galois morphism with Galois group 
\begin{equation}  \label{eqn:GOG}
G \times_{G_0,\pi} G \colonequals \{(\sigma_1,\sigma_2) : \pi(\sigma_1)=\pi(\sigma_2)\}. 
\end{equation}
\item For all $(a,b) \in (U \times_K U)(K)$ outside a thin set, the fiber $(X \times_K X)_{(a,b)} \to \Spec K$ has splitting field with Galois group $G \times_{G_0,\pi} G$.
\end{enumalph}
\end{prop}

\begin{proof}
We first prove part (a).  The fact that $X$ is an $L$-variety is a standard fact: in an affine patch as above with $B \colonequals A[x]/(f_t(x))$, we have $L \hookrightarrow \Frac(B)$ so $B \cap L$ has $\Frac(B \cap L)=L$ but contains $K$ so must be $L$, and so we get injective maps $A \hookrightarrow A \otimes_K L \hookrightarrow B$ (the second map under multiplication) giving the desired factorization.  

For (b), we briefly recall the standard proof that $X \times X$ is irreducible, taking $L=K$.  Since $X$ is irreducible there is a dense open affine domain $\Spec A \subseteq X$.  Since $K$ is algebraically closed in $K(X)$, it is algebraically closed in $A$.  We then recall that $A \otimes_K A$ is a domain (see e.g.\ the proof of Milne \cite[Proposition 4.15]{Milne-alg}).  Thus $\Spec (A \otimes_K A) \subseteq X \times X$ is irreducible, and we conclude the proof by taking the Zariski closure.  And in general if $Y$ is irreducible over $L$ then the $K$-scheme $Y_K$ is also irreducible (the topology has not changed, just the structure map).  

Next, part (c).  We first prove the statement for $K=L$.  Then $X \times_K X$ is irreducible by (b), and we want to show $\Gal(K(X \times X)\,|\,K(U \times U)) = G \times G$.  We have a diagram
\begin{equation} \label{eqn:XXUU}
\begin{aligned}
\xymatrix@R=2ex@C=2ex{
& X \times_K X \ar[dl] \ar[dr] \\
X \times_K U \ar[dr] & & U \times_K X \ar[dl] \\
& U \times_K U 
} 
\end{aligned}
\end{equation}
where the bottom maps are $G$-extensions.  By irreducibility, we get a corresponding field diagram of function fields over $K$.  By basic Galois theory (e.g. \cite[\S 14.4, Proposition 21, p.~592]{DF}), we need to show that $K(X \times U)$ and $K(U \times X)$ are linearly disjoint over $K(U \times U)$.  By the universal property of the fiber product we have 
\[ X \times X \cong (U \times X) \times_{U \times U} (X \times U) \]
as $K$-schemes (the condition that maps factor through is automatic). 
Taking generic points shows that 
\[ K(X \times X) \cong K(U \times X) \otimes_{K(U \times U)} K(X \times U); \]
since $K(X \times X)$ is a field (as $X \times X$ is irreducible), we conclude that $K(U \times X)$ and $K(X \times U)$ are linearly disjoint over $K(U\times U)$ and $K(X \times_K X)$ is their compositum.  The claim follows.

Now the general case, finishing (c).  In an open affine as in (a), the natural map is 
\[ A \otimes_K A \to B \otimes_K B \to (B \otimes_L B)_K.  \]
By (b), the scheme $X \times_L X$ is irreducible.  We consider the diagram of fields, using (a):
\begin{equation}  \label{eqn:LULX}
\begin{aligned}
\xymatrix@R=2ex@C=2ex{
& L(X \times_L X) = K((X \times_L X)_K) \\
L\bigl(X \times_L U_L\bigr) \ar@{-}[ur]  & & L\bigl(U_L \times_L X\bigr) \ar@{-}[ul] \\
& \ar@{-}[ur] L\bigl(U_L \times_L U_L\bigr) = L(U \times_K U)  \ar@{-}[ul] \\
& K(U \times_K U) \ar@{-}[u]
} 
\end{aligned}
\end{equation}
Applying the case over $L=K$, we get that the top diamond \eqref{eqn:LULX} is a compositum of linearly disjoint extensions.  Since $\Gal(L(U \times_K U)\,|\,K(U \times_K U)) \cong \Gal(L\,|\,K)$, we conclude that \eqref{eqn:GOG} holds as above by basic Galois theory.  

Part (d) follows by applying \Cref{cor:HITcor} to (c) as follows.  We have $X_a = \Spec B_a$ and $X_b = \Spec B_b$ where $B_a, B_b$ are fields, so $(X \times_K X)_{(a,b)} \cong \Spec(B_a \otimes_K B_b)$ and similarly $((X \times_L X)_K)_{(a,b)} \cong \Spec(B_a \otimes_L B_b)$.
Outside of a thin set, we have $B_a,B_b$ linearly disjoint over $L$ so $B_a \otimes_L B_b = B_a B_b$ is the compositum and $B_a \otimes_K B_b \simeq \prod_{\sigma \in \Gal(L|K)} B_aB_b$ \cite[Exercise I.2, p.~335]{FT}, showing that the two fibers have the same splitting field.  
\end{proof}


\subsection{Computation of the Galois action on \texorpdfstring{$A$}{A}}\label{sec:GalactionA}

Now let $K$ be a number field linearly disjoint from $\Q(\zeta_\ell)$ (i.e., $K \cap \Q(\zeta_\ell) = \Q$), and let $\Kbar$ be the algebraic closure.  Let $A$ be an abelian surface over $K$ as in \Cref{cons} with $(E_1, P)$ and $(E_2, Q)$ satisfying $P \in E_1[\ell](K)$ and $Q \in E_2[\ell](K)$ nontrivial. 



We recall that the modular curve $Y_1(\ell)$ is a coarse moduli space for the functor which associates to a scheme $S$ the set of isomorphism classes of pairs $(E,P)$ where $E$ is an elliptic curve over $S$ and $P \in E(S)$ is a point of order $\ell$.  For $\ell \geq 5$, in fact $Y_1(\ell)$ is a fine moduli space, so any point $[(E,P)] \in Y_1(\ell)(K)$ comes from a unique pair $(E,P)$ over $K$.  Explicitly and classically, we have the universal families in Tate normal form:
\begin{equation}
\sfE_{\ell} \colon y^2 + (1-c_\ell(t))xy - b_\ell(t)y = x^3-b_\ell(t)x
\end{equation}
where $(0,0)$ has order $\ell$, and  
\begin{equation}
\begin{aligned}
(b_5(t),c_5(t)) &=(t,t) \\
(b_7(t),c_7(t)) &=(t^3-t^2,t^2-t) 
\end{aligned}
\end{equation}
For $\ell=3$, the same conclusion holds whenever $j(E) \neq 0$ (such points have trivial stabilizer), and we may take the family 
\begin{equation} 
\sfE_{3} \colon y^2 + xy + ty = x^3 
\end{equation}
with $(0,0)$ of order $3$ and $Y_1(3) \simeq \PP^1 \smallsetminus \{0,1/27,\infty\}$ in the variable $t$.  Finally, for $\ell=2$ the generic stabilizer is $\{\pm 1\}$, so with mild resignation we choose the family
\begin{equation}
\sfE_{2} \colon y^2+xy=x^3+tx
\end{equation}
and again $(0,0)$ has order $2$ and $Y_1(2) \simeq \PP^1 \smallsetminus \{0,1/64,\infty\}$ in $t$.  

\begin{lem}\label{lem:thinset}
Let $\ell \leqslant 7$ be prime.  Then the following statements hold.
\begin{enumalph}
\item For all $t \in Y_1(\ell)(K)$ corresponding to $E=\sfE_{\ell,t}$, the image of the $\ell$-adic Galois representation
\[ \rho_{E,\ell}\colon \Gal_{K} \to \Aut(T_{\ell}(E)(\Kbar)) \simeq \GL_2(\Z_\ell) \]
is contained in 
\begin{equation} \label{eqn:surjEimg}
\left\{\begin{pmatrix}
 a & b \\ \ell c & d 
 \end{pmatrix} \in \mathrm{M}_2(\Z_\ell) : a,d \in \Z_\ell^\times, a \equiv 1 \psmod{\ell} \right\}\leqslant \GL_2(\Z_\ell)
 \end{equation}
in any basis $P_1,P_2$ for $T_\ell(E)$ such that $P_1 \bmod{\ell} = P$.  In particular, 
\[ \overline{\rho}_{E,\ell}\colon \Gal_{K} \to \Aut({E}[\ell](\Qbar)) \simeq \GL_2(\F_\ell) \]
has image contained in
\begin{equation} \label{eqn:1bd}
\left\{\begin{pmatrix}
 1 & b \\ 0 & d 
 \end{pmatrix} \in \mathrm{M}_2(\F_\ell) : d \in \F_\ell^\times \right\}\leqslant \GL_2(\F_\ell). 
 \end{equation}
\item For all $n \geq 1$, the generic Galois group of $\sfE_\ell[\ell^n] \to Y_1(\ell) \subseteq \PP^1_t$ is equal to $G_{\ell^n}$, the reduction of \eqref{eqn:surjEimg} modulo $\ell^n$.  Moreover, the algebraic closure of $K$ in $K(\sfE_\ell[\ell^n])$ is equal to $K(\zeta_{\ell^n})$.  
\item Outside of a thin set in $Y_1(\ell)(K)$, the image $\rho_{E,\ell}(\Gal_K)$ in \textup{(a)} 
is the entire subgroup in \eqref{eqn:surjEimg}.  
\end{enumalph}
\end{lem}

\begin{proof}
Part (a) follows by a direct calculation.  

For part (b), for each $n \geq 1$ we will apply Hilbert irreducibility in the form of \Cref{cor:HITcor} applied to the cover $Y_{\textup{full}}(\ell^n) \to Y_1(\ell)$ where $Y_{\textup{full}}(\ell^n)$ is the modular curve parametrizing full level $\ell^n$ structure (see Deligne--Rapoport \cite[IV-3.1]{DR} or Katz--Mazur \cite[\S 3.1]{KM}).  
The same calculation in (a) shows that the generic Galois group is contained in $G_{\ell^n}$, the reduction of \eqref{eqn:surjEimg} modulo $\ell^n$; we need to show that equality holds.

We show equality by computing degrees.  On base change to $L=K(\zeta_{\ell^n})$, we have $Y_{\textup{full}}(\ell^n)_L$ a disjoint union of components with a simply transitive action of $\Gal(L\,|\,K) \simeq (\Z/\ell^n\Z)^\times$ (since $K \cap \Q(\zeta_{\ell^n})=K$), indexed by the value taken by the Weil pairing on the basis of $\ell^n$-torsion.  All components become isomorphic upon further base change to $\C$ to the usual modular curve $Y(\ell^n)(\C)$ (rescaling the basis).  The fundamental group of the cover $Y(\ell^n)(\C) \to Y_1(\ell)(\C)$ is $G_{\ell^n} \cap \SL_2(\Z/\ell^n\Z)$ (as the corresponding quotients of the upper half-plane, which is simply connected).  This shows that the image of the generic Galois group modulo $\ell^n$ has size at least 
\[ \varphi(\ell^n)[\SL_2(\Z/\ell^n\Z):G_{\ell^n} \cap \SL_2(\Z/\ell^n \Z)]=[\GL_2(\Z/\ell^n \Z):G_{\ell^n}] \]
so equality must hold.  This argument proves the second moreover clause as well.

For part (c), we first apply \Cref{cor:HITcor} (Hilbert irreducibility) to $\sfE[\ell^n] \to Y_1(\ell)$ with $n=2$ for $\ell \geq 3$ and $\ell^n=8$ for $\ell=2$.  The result then holds after we recall the simple inductive lemma that if $H \leq \GL_2(\Z_\ell)$ is a closed subgroup such that 
\[ H \cap \left(1+\ell^{n-1}\mathrm{M}_2(\Z/\ell\Z)\right) = \left(1+\ell^{n-1}\mathrm{M}_2(\Z/\ell\Z)\right) \leq \GL_2(\Z/\ell^{n}\Z), \]
then $H$ is the full preimage of its image under natural projection $\GL_2(\Z_\ell) \to \GL_2(\Z/\ell^n\Z)$.  Indeed, computing degrees this follows from the fact the $\ell$th power map gives a surjection 
$1+\ell^m\mathrm{M}_2(\Z/\ell\Z) \to 1+\ell^{m+1}\mathrm{M}_2(\Z/\ell\Z)$ (by the binomial formula) for all $m \geq n-1$.
\end{proof}

\begin{remark}
In fact, all we need in our construction is to have surjectivity modulo $\ell^2$: see \Cref{prop:Amatrices}.  
\end{remark}

%

Choose a basis $\{P_1, P_2, Q_1, Q_2\}$ for $T_\ell(  E_1 \times E_2 )\simeq \Z_\ell^4$ as in \Cref{lem:thinset}:
\begin{itemize}
    \item $P_1 \bmod \ell =P  \in E_1[\ell](K)$,
    \item $Q_1 \bmod \ell =Q \in E_2[\ell](K)$,
    \item $\{P_1,P_2\}$ is a symplectic basis for $T_\ell E_1$, and
    \item $\{Q_1,Q_2\}$ is a symplectic basis for $T_\ell E_2$.
\end{itemize}
Then the Galois action on $( E_1\times E_2)[\ell](\Kbar)$ has image contained in the subgroup
\begin{equation}\label{eq:thebigsubmodell}
\left\{ \begin{pmatrix} 1 & b_1 & 0 & 0 \\ 0 & d_1 & 0 & 0 \\ 0 & 0 & 1 & b_2 \\ 0 & 0 & 0 & d_2 \end{pmatrix} \in \mathrm{M}_4(\F_\ell) : a_1, d_1, a_2, d_2 \in \F_\ell^\times \right\}\leqslant \GL_4(\F_{\ell}).
\end{equation}
The Galois equivariance of the Weil pairing (given explicitly by the determinant) \cite[section III.8]{SilvermanAEC}, which holds for any $\ell^n$-torsion points (or more generally on $T_\ell(E)$), further implies that \(\rho_{E_1\times E_2,\ell}(\Gal_K)\) is contained in
\begin{equation}\label{eq:thebigsub}
\left\{ \begin{pmatrix} a_1 & b_1 & 0 & 0 \\ \ell c_1 & d_1 & 0 & 0 \\ 0 & 0 & a_2 & b_2 \\ 0 & 0 & \ell c_2 & d_2 \end{pmatrix}\in \mathrm{M}_4(\Z_\ell) : 
\begin{minipage}{27ex}
$a_1, d_1,a_2,d_2 \in \Z_\ell^\times$, \\ 
$a_1 \equiv a_2 \equiv 1 \psmod{\ell}$, and \\
$a_1d_1-\ell b_1c_1 = a_2d_2-\ell b_2c_2$
\end{minipage}
\right\}\leqslant \GL_4(\Z_{\ell}). 
\end{equation}

We now show that there are infinitely many pairs where the image in fact surjects onto this group.  



\begin{prop} \label{prop:wegotitall_1}
Let $\ell \leqslant 7$ be prime.
There are infinitely many pairs $E_1,E_2$ of elliptic curves satisfying both of the following:
\begin{enumroman}
\item The image of $\rho_{E_1 \times E_2,\ell}$ is the subgroup \eqref{eq:thebigsub}; in particular, there exist points $P \in  E_1[\ell](K)$ and $Q \in E_2[\ell](K)$ of order $\ell$; and
\item $E_1$ is not geometrically isogenous to $E_2$.
\end{enumroman}
Moreover, the products $E_1 \times E_2$ fall into infinitely many distinct geometric isogeny classes.
\end{prop}


\begin{proof}

To satisfy condition (i), analogous to \Cref{lem:thinset}(b), we now apply HIT (\Cref{thm:HIT}), taking the cover $\sfE_\ell[\ell^n] \times \sfE_\ell[\ell^n] \to Y_1(\ell) \times Y_1(\ell)$.  For each factor, by \Cref{lem:thinset}(b) we know that the generic Galois group is $G_{\ell^n}$, and the constant subextension is given by $L=K(\zeta_{\ell^n})$.  The result then follows by applying \Cref{prop:G0G}(d).  


To further satisfy (ii), let $E_1 \times E_2$ have large image as in (i), and suppose that there is an isogeny $\phi \colon E_1 \to E_2$ over a finite extension $K' \supseteq K$.  Then the image of the restriction of $\rho_{E_1 \times E_2,\ell}$ to $\Gal_{K'}$ has finite index in $\rho_{E_1 \times E_2,\ell}(\Gal_K)$.  And moreover the Tate module $T_\ell(E_{2,K'})$ is conjugate to $T_\ell(E_{1,K'})$ inside the latter's Tate representation $V_\ell(E_{1,K'}) \colonequals T_\ell(E_{1,K'}) \otimes \Q_\ell$.  In particular, the image of $\rho_{(E_1 \times E_2)_{K'},\ell}$ is contained in the subgroup 
\[ \left\{\begin{pmatrix} A & 0 \\ 0 & PAP^{-1} \end{pmatrix} : A \in \GL_2(\Z_\ell) \right\} \leq \GL_4(\Z_\ell) \]
where $P$ is the matrix of $\phi_\ell \colon T_\ell(E_{1,K'}) \to T_\ell(E_{2,K'})$.  This subgroup has infinite index in the subgroup given by (b), a contradiction.  

The final statement follows quite a bit more generally, see Cantoral-Farf\'an--Lombardo--Voight \cite[Proposition 6.6.1]{FLV}.  In fact, we claim that even for fixed $E_1$, the curves $E_2$ 
that satisfy (i) and (ii)
fall into infinitely many distinct geometric isogeny classes.  We give a proof of this stronger statement.  Recall (Tate's algorithm) that $E$ has bad potentially multiplicative reduction at a prime $\frakp$ if and only if $\ord_\frakp(j(E))<0$ has negative valuation.  Let $t$ be a parameter on $Y_1(\ell)$.  We conclude in the style of Euclid: for any finite set $\{(E'_i,P'_i)\}_i \subset Y_1(\ell)(K)$ corresponding to $t_i \in K$, we can find $\frakp$ such that $\ord_\frakp(j(E'_i)) \geq 0$ and there exists $t^* \in K$ giving $(E^*,P^*) \in Y_1(\ell)(K)$ such that $\ord_\frakp(j(E_{t^*}))<0$.  Indeed, this is determined by congruence conditions on the numerator and denominator, and the resulting set has positive density so cannot be contained in the thin set excluded by Hilbert irreducibility.  If $(E^*,P^*)$ has $j(E^*)=j(t^*)$ then $E^*$ cannot be geometrically isogenous to any $E_i'$, since each $E_i'$ has potentially good reduction at $\frakp$ whereas $E^*$ has bad potentially multiplicative reduction.  It follows that the set of products $E_1 \times E_2$ of this form fall into infinitely many distinct geometric isogeny classes: 
indeed, since none of $E_1$, $E_2$ and $E_2'$ are geometrically isogenous, by the Poincar\'e Complete Reducibility Theorem, $E_1 \times E_2$ is not geometrically isogeneous to $E_1 \times E_2'$.
\end{proof}

For convenience, we rewrite the elements in \eqref{eq:thebigsub} as
\begin{equation}\label{eq:matrixl2}
\begin{pmatrix} 1+x_1\ell & b_1+y_1\ell & 0 & 0 \\ w_1\ell & d+z_1\ell & 0 & 0 \\ 0 & 0 & 1+x_2\ell & b_2+y_2\ell \\ 0 & 0 & w_2\ell & d+z_2\ell \end{pmatrix}= \begin{pmatrix} A_1 & 0 \\ 0 & A_2 \end{pmatrix}
\end{equation}
where:
\begin{itemize}
\item $d \in \{1,\dots,\ell-1\}$,
\item $b_1,b_2 \in \{0,\dots,\ell-1\}$, and
\item $w_i, x_i, y_i, z_i \in \Z_\ell$,
\end{itemize}
still subject to the condition (Weil pairing) that
\begin{equation}\label{detcond} 
\det A_1 =\det A_2. 
\end{equation}

Recall that $A=(E_1\times E_2)/\langle(P,Q)\rangle$. 
Because $A$ and $E_1\times E_2$ are isogenous via $q \colon E_1 \times E_2 \to A$, we may choose the following change of coordinates matrix $M_{q,\ell}$ 
that allows us to compute the Galois action on $A$ from the action on $E_1\times E_2$:
\begin{equation}\label{eq:Mq}
  M_{q,\ell}=\begin{pmatrix} 1 & 0 & 1/\ell & 0 \\ 0 & 1 & 0 & 0 \\ 0 & 0 & 1/\ell & 0 \\ 0 & 0 & 0 & 1\end{pmatrix}.  
\end{equation}
Then $\det M_{q,\ell}^{-1}=\ell$ and maps the lattice $\mathbb{Z}_\ell^4 + \mathbb{Z}_\ell \cdot (1/\ell,1,1/\ell,1) \subset \mathbb{Q}_\ell^4$ into $\mathbb{Z}_\ell^4$.
Conjugating the elements \eqref{eq:matrixl2} above by this change of coordinates matrix (i.e.~we compute $M_{q,\ell}^{-1}G_{\ell}M_{q,\ell}$) gives
\begin{equation}\label{eq:Amodl2matrices}
\begin{pmatrix} 1+x_1\ell & b_1+y_1\ell & x_1-x_2 & -b_2-y_2\ell \\ w_1\ell & d+z_1\ell & w_1 & 0 \\ 0 & 0 & 1+x_2\ell & b_2\ell+y_2\ell^2 \\ 0 & 0 & w_2 & d+z_2\ell \end{pmatrix},
\end{equation}
with the same conditions on the variables.  See also our companion paper \cite{FHV_expository} for more on computing Galois actions on Tate modules within an isogeny class.

To get the image of $\bar\rho_{A,\ell}\colon \Gal_{K} \to \GL_4(\F_\ell)$, we reduce this subgroup modulo $\ell$, as given in the following proposition.

\begin{prop}\label{prop:Amatrices}
The image of $\bar\rho_{A,\ell}\colon \Gal_{K} \to \GL_4(\F_\ell)$ is 
\begin{equation} \label{eqn:1b1x1w1}
\left\{ \begin{pmatrix} 1 & b_1 & x_1-x_2 & -b_2 \\ 0 & d & w_1 & 0 \\ 0 & 0 & 1 & 0 \\ 0 & 0 & w_2 & d \end{pmatrix} \in \mathrm{M}_4(\F_\ell) : 
\begin{minipage}{14ex}
$d \in \F_\ell^{\times}$\\
$b_i,w_i,x_i \in\F_\ell$
\end{minipage} \right\}\leqslant \GL_4(\F_\ell).
\end{equation}
\end{prop}

\begin{proof}
We need to check that the determinant condition 
\eqref{detcond} being satisfied does not constrain our choices of variables above:
it requires that 
\[ d+(dx_1+z_1-b_1w_1)\ell+ (x_1z_1-w_1y_1)\ell^2 = d+(dx_2+z_2-b_2 w_2)\ell+(x_2z_2-w_2y_2)\ell^2. 
\]
Given $d \in \F_\ell^{\times}$, $b_i,w_i,x_i \in\F_\ell$, choose lifts of $w_i, x_i$ to $\Z_\ell$ (which we continue to call $w_i, x_i$). One can check that taking $z_2=y_1=y_2=0$ and $z_1 = (1+x_2\ell)^{-1} (dx_2 - d x_1 + b_1 w_1 - b_2 w_2)$ gives a solution to the determinant equation, and thus an element in the image of $\rho_{A,\ell}$ reducing to an element as in \eqref{eqn:1b1x1w1} with the given entries. 
\end{proof}


\subsection{Computation of the Galois action on \texorpdfstring{$A\spcheck$}{Av} via the contragredient}\label{sec:GalactionAdual1}
Next, we would like to compare this to the Galois action on $A\spcheck[\ell](\Kbar)$. To do so, we make use of the following, as indicated in the introduction.

\begin{lem}\label{lem.contra}
Given the representation \(\rho_{A,\ell}\colon \Gal_{K} \to \Aut(T_\ell A)\), there is an isomorphism $\rho_{A\spcheck,\ell}\cong \rho_{A,\ell}^{\ast} \otimes \varepsilon_\ell$, where $\rho_{A,\ell}^{\ast}$ is the dual or contragredient representation and $\varepsilon_\ell$ is the cyclotomic representation. In particular, there is an isomorphism $\bar\rho_{A\spcheck,{\ell^n}}\cong \bar\rho_{A,{\ell^n}}^{\ast}\otimes \varepsilon_\ell$ for all \(n\in \Z_{\geq 1}\). 
\end{lem}

\begin{proof}
The tautological pairing \(T_\ell A \times T_\ell A\spcheck \to \Z_\ell(1)\) is given by taking the inverse limit over $n$ of the Weil pairing \(A[\ell^n]\times A\spcheck[\ell^n] \to \mu_{\ell^n}\). This is a perfect bilinear pairing, hence non-degenerate, and so the result follows.
\end{proof}

By the Weil pairing, 
if $\bar\rho_{A,\ell}(\sigma) = M$ for $M$ as in \eqref{eq:matrixl2}, then $\varepsilon_\ell(\sigma) = d$ \cite[section III.8]{SilvermanAEC}. 
Thus, when we take the inverse transpose of matrices as in \Cref{prop:Amatrices} and scale by this factor, we get the following. 

\begin{prop}\label{prop:Adualmatrices}
If $\bar\rho_{A,\ell}(\sigma)=\begin{pmatrix} 1 & b_1 & x_1-x_2 & -b_2 \\ 0 & d & w_1 & 0 \\ 0 & 0 & 1 & 0 \\ 0 & 0 & w_2 & d \end{pmatrix}$ as in \eqref{eqn:1b1x1w1}, then 
\begin{equation} 
\bar\rho_{A\spcheck,\ell}(\sigma) = \begin{pmatrix} d & 0 & 0 & 0 \\ -b_1 & 1 & 0 & 0 \\ z_1-z_2 & -w_1 & d & -w_2 \\ b_2 & 0 & 0 & 1 \end{pmatrix} 
\end{equation}
where \( z_1-z_2 = b_1w_1-b_2w_2-dx_1+dx_2 \in \F_\ell\).
\end{prop}

\begin{proof}
This proposition follows from the explanation above, but for the $(3,1)$-entry which is
\[ b_1w_1 - b_2w_2 - dx_1 + dx_2=z_1-z_2 \]
by the determinant condition \eqref{detcond}.
\end{proof}

\subsection{Proof of the main result}\label{sec:mainproof}

We now prove \Cref{thm:main}.


\begin{proof}[Proof of \Cref{thm:main}]
Let $K$ be a number field with $K \cap \Q(\zeta_\ell) = \Q$ and $A$ an abelian surface over \(K\) as in \Cref{cons}, with the pair $E_1,E_2$ coming from the infinite set in Proposition \ref{prop:wegotitall_1}.

Let $\sigma \in \Gal_K$.  We apply \Cref{prop:Adualmatrices}.  For $\ell=2$, we check computationally that there is no $M\in \GL_4(\F_2)$ for which $M\bar\rho_{A,2}(\sigma) M^{-1}=\bar\rho_{A\spcheck,2}(\sigma)$ for all $\sigma \in \Gal_{K}$; see the \textsf{Magma} \cite{Magma} code \cite{FHVcode}. 
Hence, these representations are not isomorphic and $A[2]$ is not isomorphic to $A\spcheck[2]$ over $K$.

It remains to show that the same is true for $\ell\in \{3,5,7\}$. We claim that this can be seen directly from just the images of the representations $\bar\rho_{A,\ell}$ and $\bar\rho_{A\spcheck,\ell}$. Indeed, $A[\ell](K)\neq \emptyset$, since the first basis element is fixed by $\Gal_{K}$. However, one can check that there is no vector in $\F_\ell^4$ which is fixed by $\bar\rho_{A\spcheck,\ell}(\sigma)$ for all $\sigma \in \Gal_{K}$, so $A\spcheck[\ell](K) = \emptyset$.  More precisely, although each matrix $\bar\rho_{A\spcheck,\ell}(\sigma)$ has fixed vectors, the coordinates depend on the matrix entries, whose values are unconstrained, as shown in \Cref{prop:Adualmatrices}.  (Note that this argument fails for $\ell=2$, since both $A$ and $A\spcheck$ have a rational $2$-torsion point, given on $A\spcheck$ by the third basis element.)
\end{proof}

\subsection{Generalizing the construction}\label{sec:generalconst} 

Construction~\ref{cons} can be generalized to produce more examples of abelian surfaces satisfying Theorem~\ref{thm:main}. Here, we give the construction and outline the ways in which the results of sections~\ref{sec:construction}--\ref{sec:mainproof} need to be adapted to arrive at the result.

Instead of starting with \(E_1\) and \(E_2\) elliptic curves over \(k\) each with a $k$-rational $\ell$-torsion point, we suppose more generally that there are cyclic subgroups $C_1 \leqslant E_1[\ell]$ and $C_2 \leqslant E_2[\ell]$ such that $c \colon C_1 \xrightarrow{\sim} C_2$ are isomorphic as $\Gal_k$-modules. We then take
\[ G \colonequals \langle (P,c(P)) : P \in C_1 \rangle \leqslant E_1 \times E_2
\quad\text{and}\quad
A\colonequals (E_1\times E_2)/G. \]
When $\ell=2$ or when the Galois action on $C_1\simeq C_2$ is trivial, we recover Construction~\ref{cons}. 

For $\ell=3,5,7$, there are again infinitely many elliptic curves $E$ over $K$ with a cyclic subgroup $C \leqslant E[\ell](\Kbar)$ stable under $\Gal_K$---they are parametrized by the modular curve $Y_0(\ell)$, which is birational to $\P^1$. Moreover, for such a pair $(E,C)$, there exist infinitely many pairs $(E',C')$ such that $C \simeq C'$ as $\Gal_K$-modules. 
This can be seen by constructing a moduli space for the desired pairs $(E',C')$ as a twist of $Y_1(\ell)$.  This same strategy is employed in the construction of families of elliptic curves with a fixed mod $N$ representation (see e.g.\ Silverberg \cite{Silverberg}). 
This moduli space, $Y_C(\ell)$, has a universal family $E_{\textup{univ},C}(\ell)$ over it (or at least over an open subset).

By sourcing our elliptic curves from $Y_0(\ell)$ instead of $Y_1(\ell)$, the images of the $\ell$-adic and mod $\ell$ Galois representations will change. There is no longer a condition on $a \psmod{\ell}$ in~\eqref{eqn:surjEimg}, and similarly for the mod $\ell$ representation (that is, the $1$ in the top left entry can be any $a \in \F_\ell^\times$). The same modification must be made in~\eqref{eq:thebigsubmodell} and~\eqref{eq:thebigsub}. Then the proof of Proposition~\ref{prop:wegotitall_1} goes through the same, replacing $\mathsf{E}_{C,\ell}$ with a versal family over $Y_C(\ell)$. Thus there are again infinitely many pairwise geometrically non-isogenous such abelian surfaces constructed as above with maximal Galois image.

Finally, the images of $\bar{\rho}_{A,\ell}$ and $\bar{\rho}_{A\spcheck,\ell}$ can be calculated using the same techniques as in sections~\ref{sec:GalactionA}--\ref{sec:GalactionAdual1}.  For $\ell=3,5,7$, it will no longer be the case that $A[\ell](K)\neq \emptyset$; rather, both $A[\ell](\Kbar)$ and $A\spcheck[\ell](\Kbar)$ have a unique Galois-stable line.  The Galois actions on these lines can be shown to differ---more precisely, if $\chi$ is the character on the common Galois-stable line in $E_1[\ell]$ and $E_2[\ell]$, then the action on the line in $A[\ell]$ is $\chi$ whereas for $A\spcheck[\ell]$ it is $\chi'$ where $\chi\chi'=\varepsilon_\ell$---and so $A[\ell]$ and $A\spcheck[\ell]$ cannot be isomorphic group schemes over $K$.

\section{Further analysis and discussion} \label{sec:theproofs}

With \Cref{thm:main} now proven, we conclude with an application and some final remarks.  In \cref{sec:perm}, we examine the associated permutation representations, proving \Cref{cor:main} and giving an application to derived equivalences of Kummer fourfolds. In \cref{sec:final}, we examine the context of our results, including considering further properties the Galois actions on $A[n]$ and $A\spcheck[n]$ must or need not share, with an eye toward how our results may be extended in the future.

\subsection{Associated permutation representations}\label{sec:perm}


Following the notation in the introduction, for an abelian surface \(A\) over $K$, let $\pi_{A,\ell} \colon \Gal_{K} \to \Sym(A[\ell]) \simeq S_{\ell^{4}}$ be the permutation representation associated to \(\bar\rho_{A,\ell}\). 
The following result shows that the associated permutation and linear representations of
abelian surfaces $A$ constructed as in \Cref{cons} 
are also non-isomorphic, which proves \Cref{cor:main}. 

\begin{prop}\label{prop:maincor}
Let $A$ be an abelian surface as in \Cref{cons}, coming from a pair \(E_1, E_2\) as in \Cref{prop:wegotitall_1}.  Then for $\ell \in \{3,5,7\}$, the permutation representations \(\pi_{A,\ell}\) and \(\pi_{A\spcheck,\ell}\) are not isomorphic. Moreover, the induced linear representations over any field \(F\) with $\opchar F = 0$ are not isomorphic. 
\end{prop}

\begin{proof}
We can see this computationally in multiple ways; see the \textsf{Magma} code provided \cite{FHVcode}. For the permutation representations, we can check that the permutation characters are inequivalent. For the induced linear representations, we compute the multiplicities of the trivial representation in the induced linear representations; we find that the multiplicities are different (for $\ell=5$ and $7$, the computations are quite time-consuming!).  We check this over $\Q$, but for $F$ a field with $\opchar F =0$, the multiplicity of the trivial representation does not change under base change to $F$.
Since the induced linear representations are not isomorphic, this also shows that the permutation representations cannot be isomorphic.
\end{proof}

\begin{remark} 
We should expect that, in general, information is lost when passing from the representation \(\bar\rho_{A,\ell}\) to the permutation representation \(\pi_{A,\ell}\) and then further to the induced linear representations---in each case, there are more possible elements to conjugate by.  

For $\ell=2$ above, we check in \textsf{Magma} that the permutation representations $\pi_{A,2} \not\simeq \pi_{A\spcheck,2}$ are not isomorphic, but the linear representations are isomorphic \cite{FHVcode} (so the first but not the second statement holds in \Cref{prop:maincor} for $\ell=2$).
\end{remark}



\subsection{Final remarks} \label{sec:final}

In closing, we look at the larger context of our results. Although we have shown that the Galois action on the torsion groups of an abelian surface and its dual can be different, it is interesting to consider whether other weaker relationships hold. Then, we speculate on further constructions of abelian surfaces or abelian varieties which would satisfy the conclusion of \Cref{thm:main}.

First, we pause to prove the statement about semisimplifications made in the introduction.

\begin{lem} \label{lem:semisimp}
Let $A$ be an abelian variety over a number field $K$ and let $\ell$ be prime.  Then the semisimplifications of the mod $\ell$ Galois representations attached to $A$ and $A\spcheck$ are equivalent.
\end{lem}

\begin{proof}
Let $\lambda \colon A \to A\spcheck$ be a polarization.  Then via $\lambda$, the Tate module $T_\ell A$ is conjugate to $T_\ell A\spcheck$ as $\Gal_K$-lattices inside $V_\ell A \colonequals (T_\ell A) \otimes_{\Z_\ell} \Q_\ell$.  Therefore the reductions modulo $\ell$ have the same semisimplifications (the sum given by the factors in a Jordan--H\"older filtration).

We can also see this through point counts as follows.  For all nonzero prime ideals $\frakp$ in the ring of integers of $K$ that are of good reduction for $A$, we obtain an isogeny $\lambda_{\frakp} \colon A_{\F_\frakp} \to A\spcheck_{\F_\frakp}$ over the residue field $\F_\frakp$ between the reductions of $A$ and $A\spcheck$ modulo $\frakp$.  Hence $\rhobar_{A,\ell}(\Frob_\frakp)$ and $\rhobar_{A\spcheck,\ell}(\Frob_\frakp)$ have the same characteristic polynomials for a dense set of Frobenius elements $\Frob_\frakp \in \Gal_K$.  The traces determine the semisimplifications up to isomorphism, by the Brauer--Nesbitt theorem.
\end{proof}

The next result shows that, while the images of $\overline{\rho}_{A,n}$ and $\overline{\rho}_{A\spcheck,n}$ can differ, the kernels (and hence their fixed fields) always \emph{agree}!

\begin{lem} \label{lem:samenumberfield}
For all $n \in \Z_{\geq 1}$, we have $K(A[n])=K(A\spcheck[n]) \subset K^{\textup{al}}$. 
\end{lem}

\begin{proof}
We show that $\ker \overline{\rho}_{A,n} = \ker \overline{\rho}_{A\spcheck,n} \leqslant \Gal(K^{\textup{al}}\,|\,K)$.  Let $\sigma \in \Gal(K^{\textup{al}}\,|\,K)$.  We prove the containment $(\subseteq)$ and suppose that $\overline{\rho}_{A,n}(\sigma)=1$.  Then $\overline{\rho}_{A\spcheck,n}(\sigma)=\overline{\rho}_{A,n}(\sigma)^* \eps_n(\sigma)=1$ if and only if $\eps_n(\sigma)=1$, so we prove this.  Since $A$ has a primitive polarization $\lambda$ over $K$, there exist $P,Q \in A[n](\Kbar)$ with Weil pairing $\langle P, Q \rangle_\lambda = \zeta_n$: indeed, if the polarization is of type $(d_1,\dots,d_g)$ with corresponding Frobenius basis $P_1,\dots,P_g,Q_1,\dots,Q_g$ (i.e., $\langle P_i, Q_j \rangle_{\lambda} = d_i,0$ as $i=j$ or not), then $\gcd(d_1,\dots,d_g)=1$ so there exist $a_1,\dots,a_g \in \Z$ such that $d_1 a_1 + \dots + d_g a_g = 1$ and thus $P \colonequals \sum_{i=1}^g a_i P_i$ and $Q \colonequals \sum_{i=1}^g Q_i$ have $\langle P,Q \rangle_\lambda = \zeta_n^{\sum_i a_i d_i} = \zeta_n$.  Therefore  
\begin{equation} 
\zeta_n = \langle P, Q \rangle_\lambda = \langle \sigma(P), \sigma(Q) \rangle_\lambda = \zeta_n^{\eps_n(\sigma)} 
\end{equation}
the latter by the Galois equivariance of the pairing.  We conclude indeed that $\eps_n(\sigma)=1$.  

The containment $(\supseteq)$ follows from interchanging $A$ with $A\spcheck$ and applying the canonical isomorphism $(A\spcheck)\spcheck \cong A$.  
\end{proof}



\begin{remark}
The subgroups $H \colonequals \img \overline{\rho}_{A,\ell}$ and $H' \colonequals \img \overline{\rho}_{A\spcheck,\ell}$ are subgroups of the subgroup $G \leqslant \GL_4(\F_\ell)$ of matrices preserving a rank $2$ alternating form.  (See also \Cref{prop:listofsubgroups}.) 
Recall that two subgroups $H,H' \leqslant G$ are \defi{Gassmann equivalent} if $\#(H \cap C) = \#(H' \cap C)$ for all conjugacy classes $C$ in $G$.  We calculate that for $\ell=2$, in fact the images are not Gassmann equivalent.  
\end{remark}

It is interesting to consider which subgroups of $\GL_4(\F_\ell)$ could be the image of 
$\overline{\rho}_{A,\ell}$ if $\overline{\rho}_{A,\ell}$ and $\overline{\rho}_{A\spcheck,\ell}$ are not isomorphic.
In the following result, we enumerate such possible Galois images in the case of $\ell=2$.

\begin{prop} \label{prop:listofsubgroups}
The following statements hold.
\begin{enumalph}
\item The subgroup $G \leqslant \GL_4(\F_2)$ of elements preserving (up to scaling) the unique rank $2$ degenerate symplectic form is a solvable group of order $576$ and exponent $12$ isomorphic to $C_2^4 \rtimes S_3^2$ as a group.  
\item Of the $128$ conjugacy classes of subgroups $H \leqslant G$, there are $52$ for which the natural inclusion $H \hookrightarrow G \leqslant \GL_4(\F_2)$ is not conjugate (in $\GL_4(\F_2)$) to its (twisted) contragredient, and a further $26$ for which the groups are conjugate but the corresponding representations are not equivalent.  
\end{enumalph}
\end{prop}

For the additional 26 groups in part (b), there exists $g \in \GL_4(\F_2)$ such that $g^{-1} H g = H\spcheck$ but there is no $g \in \GL_4(\F_2)$ such that $g^{-1} h g = h\spcheck$ for all $h \in H$.  

\begin{proof}
This follows from a direct calculation with matrix groups, which was performed in \textsf{Magma}; see the code \cite{FHVcode}. 
\end{proof}

The list of groups from \Cref{prop:listofsubgroups}(b) is already quite interesting: the smallest group has size $4$, the largest has index $2$ in $G$!


We conclude with a few final comments on constructing abelian surfaces.

First, Bruin \cite{Bruin} has exhibited algorithms to work with finite flat group schemes; using these methods, we could exhibit specific instances of our construction (including the Galois action).  In the same vein, although our abelian surfaces are not principally polarized, so cannot arise as Jacobians of genus $2$ curves, they may still be obtained as the Prym variety attached to a cover of curves.  It would be interesting to see this explicitly, for example in the case $\ell=2$ \cite{HSS}.

Second, abelian varieties with real multiplication over fields with nontrivial narrow class group also give potential examples of abelian varieties without principal polarizations which could be used as input into our method.  The underlying parameter space is now a Hilbert modular variety which may be disconnected---only one generic component corresponds to those with a principal polarization.  

Third, given that our construction is limited to $\ell \leqslant 7$, one may wonder when it is even possible to construct families of abelian varieties of dimension $g$ defined over an open $U \subseteq \PP^n$ with a polarization of minimal degree $d > 1$.  Even the existence of a \emph{single} such abelian variety is constrained: for abelian varieties $A$ with fixed dimension $g$ over number fields of fixed degree $d$, the minimal degree of a polarization on $A$ is conjecturally bounded by a constant $c(d,g)$, see R\'emond \cite[Th\'eor\`eme 1.1(1)]{Remond} who deduces this finiteness from Coleman's conjecture on endomorphism algebras using Zarhin's trick.  In particular, if $g$ and $d$ are fixed, then we can only have $A[\ell] \not\simeq A\spcheck[\ell]$ for $\ell \leq c(d,g)$.



\begin{thebibliography}{FLV23+}


\bibitem[AKW17]{AKW}
Benjamin Antieau, Daniel Krashen, and Matthew Ward,
\emph{Derived categories of torsors for abelian schemes}, Adv.~ Math.\ \textbf{306} (2017), 1--23.


\bibitem[BS23]{BS}
Pawel Bor\'owka and Anatoli Shatsila, \emph{Hyperelliptic genus 3 curves with involutions and a Prym map}, 2023, preprint, \texttt{arXiv:2308.07038}.

\bibitem[BCP97]{Magma}
W.~Bosma, J.~Cannon, and C.~Playoust, \emph{The Magma algebra system.\ I.\ The user language}, J.\ Symbolic Comput.\ \textbf{24} (3--4), 1997, 235--265.

\bibitem[Bru17]{Bruin}
Peter Bruin, \emph{Dual pairs of algebras and finite commutative group schemes}, 2017, preprint, \texttt{arXiv:1709.09847}.  

\bibitem[CP10]{MO}
Brian Conrad and Bjorn Poonen, \emph{Non-principally polarized complex abelian varieties}, 2010, \url{https://mathoverflow.net/q/17014}.

\bibitem[DR73]{DR} 
P.~Deligne and M.~Rapoport, \emph{Les sch\'emas de modules de courbes elliptiques}, Modular functions of one variable, II (Proc. Internat. Summer School, Univ. Antwerp, Antwerp, 1972), 1973, Lecture Notes in Math., vol.~349, 143--316.

\bibitem[DF04]{DF}
David S.~Dummit and Richard M.~Foote, \emph{Abstract algebra}, 3rd.\ ed., John Wiley \& Sons, Inc., Hoboken, NJ, 2004.


\bibitem[FH23]{Frei-Honigs}
Sarah Frei and Katrina Honigs, \emph{Groups of symplectic involutions on symplectic varieties of Kummer type and their fixed loci}, Forum of Math. Sigma \textbf{11}, 2023, E40.

\bibitem[FHV23]{FHVcode}
Sarah Frei, Katrina Honigs, and John Voight,
\emph{Code accompanying ``On abelian varieties whose torsion is not self-dual''}, 2023. \url{https://github.com/sjfrei/FHV-abeliansurfaces}.

\bibitem[FHV25]{FHV_expository}
Sarah Frei, Katrina Honigs, and John Voight, \emph{A framework for Tate modules of abelian varieties under isogeny}, 2025, preprint, \texttt{arXiv:2503.12756}, to appear in {\it Essent. Number Theory}.

\bibitem[FLV23+]{FLV}
Victoria Cantoral-Farf\'an, Davide Lombardo, and John Voight, \emph{Monodromy groups of Jacobians with definite quaternionic multiplication}, 2023, preprint, \texttt{arXiv:2203.08593}.

\bibitem[FT93]{FT}
A.\ Fr\"ohlich and M.~J.\ Taylor, \emph{Algebraic number theory}, Cambridge Stud.\ Adv.\ Math., vol.~27, Cambridge University Press, Cambridge, 1993.

\bibitem[Gor02]{Goren}
Eyal Z.\ Goren, \emph{Lectures on Hilbert modular varieties and modular forms}, CRM Monograph Ser., vol.~14, Amer.\ Math.\ Soc., Providence, RI, 2002.


\bibitem[HSS21]{HSS}
Jeroen Hanselman, Sam Schiavone, and Jeroen Sijsling, \emph{Gluing curves of genus 1 and 2 along their 2-torsion}, Math.~Comp.~\textbf{90} (2021), no.~331, 2333--2379.

\bibitem[HT13]{Hassett-Tschinkel}
Brendan Hassett and Yuri Tschinkel, \emph{Hodge theory and {L}agrangian planes on generalized {K}ummer fourfolds}, Mosc.~Math.~J.\ \textbf{13} (2013), no.~1, 33--56, 189.  


\bibitem[Huy19]{Huybrechts}
Daniel Huybrechts,
\emph{Motives of isogenous {K}3 surfaces},
Comment.~Math.~Helv.\ 
\textbf{94} (2019), no.~3, 445--458.

\bibitem[KM85]{KM}
Nicholas M.~Katz and Barry Mazur, \emph{Arithmetic moduli of elliptic curves}, Ann.\ Math.\ Stud., vol.~108, Princeton University Press, Princeton, NJ, 1985.



\bibitem[Lang83]{Lang}
Serge Lang, \emph{Fundamentals of Diophantine geometry}, Springer, New York, 1983.






\bibitem[Mil12]{Milne-alg}
James S.~Milne, \emph{Algebraic geometry (v5.22)}, 2012, available at \url{http://www.jmilne.org/math/}.

\bibitem[Muk81]{Mukai}
Shigeru Mukai, \emph{Duality between $D(X)$ and $D(X\spcheck)$ with its application to Picard sheaves}, Nagoya Math. J.81(1981), 153--175.




\bibitem[R\'em18]{Remond}
Ga\"el R\'emond, \emph{Conjectures uniformes sur les vari\'et\'es ab\'eliennes}, Q.\ J.\ Math.\ \textbf{69} (2018), no.~2, 459--486.




\bibitem[Sal82]{Saltman}
David J.~Saltman, \emph{Generic Galois extensions and problems in field theory}, Adv.\ Math. \textbf{43} (1982), no.~3, 250--283.

\bibitem[Ser92]{serre2} 
Jean-Pierre Serre, \emph{Topics in Galois theory}, Research Notes in Mathematics, vol.~1, Jones and Bartlett, Boston, MA, 1992.

\bibitem[Ser97]{Serre-lect}
Jean-Pierre Serre, \emph{Lectures on the Mordell-Weil theorem}, Aspects Math., Friedr.\ Vieweg \& Sohn, Braunschweig, 1997.

\bibitem[Sil97]{Silverberg}
Alice Silverberg, \emph{Explicit families of elliptic curves with prescribed mod $N$ representations}, Modular forms and Fermat's last theorem, Springer-Verlag, New York, 1997, 447--461.

\bibitem[Sil09]{SilvermanAEC}
Joseph H.\ Silverman, \emph{The arithmetic of elliptic curves}, 2nd ed., Grad.\ Texts in Math., vol.~106, Springer, Dordrecht, 2009.




\bibitem[Wit24]{Wittenberg}
Olivier Wittenberg, \emph{Park City lecture notes: around the inverse Galois problem}, IAS/Park City Mathematics Series, AMS, to appear.


\bibitem[Zyw23]{Zywina}
David Zywina, \emph{Families of abelian varieties and large Galois images},  Int.\ Math.\ Res.\ (2002), published online, 1--58.  

\end{thebibliography}
\end{document}